\documentclass{amsart}

\usepackage{amsmath}
\usepackage{amsthm,amsfonts,amssymb,mathrsfs}
\usepackage{epic,eepic}
\usepackage{yfonts}
\usepackage{paralist,enumerate}
\usepackage[all]{xy}

\newtheorem{theorem}{Theorem}
\newtheorem{proposition}[theorem]{Proposition}
\newtheorem{lemma}[theorem]{Lemma}
\newtheorem{corollary}[theorem]{Corollary}
\newtheorem{conjecture}[theorem]{Conjecture}

\newtheorem*{Brunn-Minkowski}{Brunn-Minkowski Inequality}
\newtheorem*{proposition*}{Proposition}
\newtheorem*{conjecture*}{Conjecture}
\newtheorem*{question}{Question}

\theoremstyle{definition}
\newtheorem{definition}[theorem]{Definition}

\theoremstyle{remark}
\newtheorem{remark}[theorem]{Remark}
\newtheorem{example}[theorem]{Example}

\usepackage{ifpdf}
\ifpdf
  \usepackage{graphicx}
  \usepackage{xcolor}
\else
  \usepackage[dvipdfmx]{graphicx}
  \usepackage[dvipdfmx]{xcolor}
\fi

\begin{document}

\title[Projective hypersurfaces and chromatic polynomial of graphs]{Milnor numbers of projective hypersurfaces \\ and the chromatic polynomial of graphs}
\author{June Huh}
\address{Department of Mathematics \\ University of Illinois\\
  Urbana \\ IL 61801\\ USA}
\email{huh14@illinois.edu}
\curraddr{Department of Mathematics \\ University of Michigan \\
Ann Arbor \\ MI 48109 \\ USA}
\email{junehuh@umich.edu}
\thanks{The author acknowledges support from National Science Foundation grant DMS 0838434 ``EMSW21-MCTP: Research Experience for Graduate Students''.}

\maketitle

\section{Introduction}

George Birkhoff introduced a function $\chi_G(q)$, defined for all positive integers $q$ and a finite graph $G$, which counts the number of proper colorings of $G$ with $q$ colors. As it turns out, $\chi_G(q)$ is a polynomial in $q$ with integer coefficients, called the \emph{chromatic polynomial} of $G$. 

Recall that a sequence $a_0,a_1,\ldots,a_n$ of real numbers is said to be \emph{unimodal} if for some $0 \le i \le n$,
\[
a_0 \le a_1 \le \cdots \le a_{i-1} \le a_i \ge a_{i+1} \ge \cdots \ge a_n,
\]
and is said to be \emph{log-concave} if for all $0 < i < n$,
\[
a_{i-1}a_{i+1}\le a_i^2.
\]
We say that the sequence has \emph{no internal zeros} if the indices of the nonzero elements are consecutive integers. Then in fact a nonnegative log-concave sequence with no internal zeros is unimodal. Like many other combinatorial invariants, the chromatic polynomial shows a surprising log-concavity property. See \cite{Aigner,Brenti,Stanley,Stanley2} for a survey of known results and open problems on log-concave and unimodal sequences arising in algebra, combinatorics, and geometry. The goal of this paper is to answer the following question concerning the coefficients of chromatic polynomials.

\begin{conjecture}\label{Read}
Let $\chi_G(q)=a_n q^n - a_{n-1} q^{n-1} + \cdots + (-1)^n a_0$ be the chromatic polynomial of a graph $G$. Then the sequence $a_0,a_1,\ldots,a_n$ is log-concave.
\end{conjecture}

Read \cite{Read} conjectured in 1968 that the above sequence is unimodal. Soon after Rota, Heron, and Welsh formulated the conjecture in a more general context of matroids \cite{Rota, Heron, Welsh}.  Let $M$ be a matroid and $\mathscr{L}$ be the lattice of flats of $M$ with the minimum $\hat{0}$. The \emph{characteristic polynomial} of $M$ is defined to be
\[
\chi_M(q)=\sum_{x \in \mathscr{L}} \mu(\hat{0},x) q^{\text{rank}(M)-\text{rank}(x)},
\]
where $\mu$ is the \emph{M\"obius function} of $\mathscr{L}$. We refer to \cite{Oxley, Welsh} for general background on matroids.

\begin{conjecture}\label{Rota-Welsh}
Let $\chi_M(q)=a_n q^n - a_{n-1} q^{n-1} + \cdots + (-1)^n a_0$ be the characteristic polynomial of a matroid $M$. Then the sequence $a_0,a_1,\ldots,a_n$ is log-concave.
\end{conjecture}

Recent work of Stanley \cite[Conjecture 3]{Stanley} \cite[Problem 25]{Stanley2} has renewed interest in the above conjectures. We will show in Corollary \ref{main} that Conjecture \ref{Rota-Welsh} is valid for matroids representable over a field of characteristic zero. 

\begin{theorem}
If $M$ is representable over a field of characteristic zero, then the coefficients of the characteristic polynomial of $M$ form a sign-alternating log-concave sequence of integers with no internal zeros.
\end{theorem}

If $M$ is the cycle matroid of a simple graph $G$, then
$\chi_G(q) = q^c \chi_M(q)$,
where $c$ is the number of connected components of the graph. Since graphic matroids are representable over every field,  this implies the validity of Conjecture \ref{Read}. The approach of the present paper can be viewed as following two of Rota's ideas \cite{Kung, StanleyRota}: first, the idea that the values of the M\"obius function should be interpreted as an Euler characteristic; second, the idea that the log-concavity of such quantities should come from their relation with quermassintegrals, or more generally, mixed volumes of convex bodies.

One of the most important numerical invariants of a germ of an analytic function $f : \mathbb{C}^n \to \mathbb{C}$ with an isolated singularity at the origin is the sequence $\big\{\mu^i(f)\big\}_{i=0}^n$ introduced by Teissier \cite{Teissier}. Algebraically, writing $J_f$ for the ideal generated by the partial derivatives of $f$, the sequence $\big\{\mu^i(f)\big\}_{i=0}^n$ is defined by saying that $\dim_\mathbb{C} \mathbb{C}\{x_1,\ldots,x_n\}/\mathfrak{m}^u J_f^v$ is equal to a polynomial
\[
\frac{\mu^0(f)}{n!} u^{n} + \cdots + \frac{\mu^i(f)}{(n-i)i!} u^{n-i}v^i+ \cdots + \frac{\mu^{n}(f)}{n!} v^{n} + \text{(lower degree terms)}
\]
for large enough $u$ and $v$. Geometrically, $\mu^i(f)$ is the Milnor number of $f|_H$, where $H$ is a general $i$-dimensional plane passing through the origin of $\mathbb{C}^n$. Like any other mixed multiplicities of a pair of $\mathfrak{m}$-primary ideals in a local ring, $\big\{\mu^i(f)\big\}_{i=0}^n$ form a \emph{log-convex} sequence \cite[Example 3]{Teissier2}.

Let $h$ be any nonconstant homogeneous polynomial in $\mathbb{C}[z_0,\ldots,z_n]$. In analogy with \cite{Teissier}, we define a sequence $\big\{\mu^i(h)\big\}_{i=0}^n$ by saying that $\dim_\mathbb{C} \mathfrak{m}^uJ_h^v/\mathfrak{m}^{u+1}J_h^v$ is equal to a polynomial
\[
\frac{\mu^0(h)}{n!} u^{n} + \cdots + \frac{\mu^i(h)}{(n-i)i!} u^{n-i}v^i+ \cdots + \frac{\mu^{n}(h)}{n!} v^{n} + \text{(lower degree terms)}
\]
for large enough $u$ and $v$. Theorem \ref{A} identifies $\mu^i(h)$ with the number of $i$-cells in a CW model of the complement $D(h)$ of the hypersurface in $\mathbb{P}^n$ defined by $h$. This CW model of $D(h)$ is interesting in view of the following facts:
\begin{enumerate}[\noindent (1)]
\item The sequence $\mu^i(h)$ is related to the \emph{Chern-Schwartz-MacPherson class} \cite{MacPherson} of the hypersurface via the formula
\[
c_{SM}\big({\mathbf 1}_{D(h)}\big) = \sum_{i=0}^n (-1)^i \mu^i(h) H^i (1+H)^{n-i},
\]
where $H^i$ is the class of a codimension $i$ linear subspace in the Chow ring $A_*(\mathbb{P}^n)$. The above is a reformulation in terms of $\mu^i(h)$ of a formula due to Aluffi \cite[Theorem 2.1]{Aluffi}. See Remark \ref{CSM-class}.
\item If $h$ is a product of linear forms, then $\mu^i(h)$ are the Betti numbers of $D(h)$. In this case, the sequence $\mu^i(h)$ is determined by the expression
\[
\chi_M(q)/(q-1) = \sum_{i=0}^n (-1)^i \mu^i(h) q^{n-i},
\]
where $M$ is the matroid corresponding to the central hyperplane arrangement in $\mathbb{C}^{n+1}$ defined by $h$. This CW model of $D(h)$ is the one used by Dimca and Papadima in \cite{Dimca-Papadima} to show that $D(h)$ is \emph{minimal}. See Corollary \ref{formula}.
\end{enumerate}

Theorem \ref{bound} is an analogue of Kouchnirenko's theorem \cite{Kouchnirenko} relating the Milnor number with the Newton polytope. Let $\Delta \subset \mathbb{R}^n$ be the standard $n$-dimensional simplex, and let $\Delta_h \subset \mathbb{R}^n$ be the convex hull of exponents of dehomogenized monomials appearing in one of the partial derivatives of $h$. Then the numbers $\mu^i(h)$ satisfy
\[
b_i\big(D(h)\big) \le \mu^i(h) \le \text{MV}_n(\underbrace{\Delta,\ldots,\Delta}_{n-i}, \underbrace{\Delta_h,\ldots,\Delta_h}_i),
\]
where $\text{MV}_n$ stands for the mixed volume of convex polytopes. Therefore, if $h$ has small Newton polytope under some choice of coordinates, then the Betti numbers of $D(h)$ cannot be large. Example \ref{sharp} shows that for each $n$ there is an $h$ for which the equalities hold simultaneously for all $i$.

Theorem \ref{B} characterizes homology classes corresponding to subvarieties of $\mathbb{P}^n \times \mathbb{P}^m$, up to a positive integer multiple. Let $\xi$ be an element of the Chow group,
\[
\xi =\sum_{i} e_i \big[\mathbb{P}^{k-i} \times \mathbb{P}^{i}\big] \in A_k(\mathbb{P}^n \times \mathbb{P}^m),
\]
where the term containing $e_i$ is zero if $n<k-i$ or $m<i$. Then some multiple of $\xi$ corresponds to an irreducible subvariety iff the $e_i$ form a log-concave sequence of nonnegative integers with no internal zeros. In particular, the numbers $\mu^i(h)$ form a log-concave sequence of nonnegative integers with no internal zeros for any $h$. Combined with Corollary \ref{formula}, this shows the validity of Conjecture \ref{Rota-Welsh} for matroids representable over $\mathbb{C}$.

Trung and Verma show in \cite{Trung-Verma} that the mixed volumes of lattice polytopes in $\mathbb{R}^n$ are the mixed multiplicities of certain monomial ideals, each generated by monomials of the same degree. They then ask whether an analogue of the Alexandrov-Fenchel inequality on mixed volumes of convex bodies holds for ideals of height $n$ in a local or standard graded ring of dimension $n+1$ \cite[Question 2.7]{Trung-Verma}. In Example \ref{counter-example}, we show that the answer to their question is no in general. 

As will be clear from below, the present work is heavily indebted to other works in singularity theory, algebraic geometry, and convex geometry. Specifically, Theorem \ref{A} depends on a theorem of Dimca and Papadima \cite{Dimca-Papadima}. 
Theorem \ref{B} is a small variation of a theorem of Teissier and Khovanskii, combined with a result of Shephard \cite{Khovanskii,Shephard,Teissier3}.

\section{Preliminaries}

\subsection{Mixed multiplicities of ideals}

We give a quick introduction to mixed multiplicities of ideals. We refer to \cite{Trung,Trung-Verma} for a fuller account. 

Let $S$ be a local or standard graded algebra, $\mathfrak{m}$ be the maximal or the irrelevant ideal, and $J$ be an ideal of $S$. We define the standard bigraded algebra $R$ by
\[
R=R(\mathfrak{m} | J) = \bigoplus_{(u,v) \in \mathbb{N}^2} \mathfrak{m}^u J^v / \mathfrak{m}^{u+1} J^v, \qquad \text{where} \quad \mathfrak{m}^0 = J^0 = S.
\]
Then $R$ has a Hilbert polynomial, meaning that there is a polynomial $\text{HP}_{R}$ such that
\[
\text{HP}_{R}(u,v) = \dim_{S/\mathfrak{m}} \mathfrak{m}^u J^v / \mathfrak{m}^{u+1} J^v \quad \text{for large $u$ and large $v$}.
\]
We write
\[
\text{HP}_{R}(u,v) = \sum_{i=0}^n \frac{e_i}{(n-i)! i!} u^{n-i} v^i + \text{(lower degree terms)},
\]
where $n$ is the degree of $\text{HP}_R$. Then $e_i=e_i(\mathfrak{m}|J)$ are nonnegative integers, called the \emph{mixed multiplicities} of $\mathfrak{m}$ and $J$. If $J$ is an ideal of positive height, then $n = \dim S -1$  \cite[Theorem 1.2]{Trung-Verma}.

\begin{remark}\label{projective degree}
Let $S$ be the homogeneous coordinate ring of a projective variety $X \subseteq \mathbb{P}^n$ over a field $\mathbb{K}$, $J \subset S$ be an ideal generated by nonzero homogeneous elements of the same degree $h_0,h_1,\ldots,h_m$, and $\varphi_J$ be the rational map
\[
\varphi_J : X \dashrightarrow \mathbb{P}^m
\]
defined by the ratio $(h_0:h_1:\cdots:h_m)$. By the \emph{graph} of $\varphi_J$ we mean the closure $\Gamma_J$ in $\mathbb{P}^n \times \mathbb{P}^m$ of the graph of $\varphi_J|_U$, where $U$ is an open subset of $X$ where $\varphi_J$ is defined.
Note that a bihomogeneous polynomial $f$ in the variables $\{x_i,y_j\}_{0 \le i \le n,0 \le j \le m}$ vanishes on $\Gamma_J$ iff it has zero image in $R$, that is, $R$ is the bihomogeneous coordinate ring of $\Gamma_J$. It follows from the standard relation between Hilbert polynomials and intersection theory that
\[
\big[\Gamma_J\big] = \sum_i e_i \big[ \mathbb{P}^{k-i} \times \mathbb{P}^i \big] \in A_{k}(\mathbb{P}^n \times \mathbb{P}^m),
\]
where $k$ is the dimension of $X$.
In other words, the mixed multiplicities of $\mathfrak{m}$ and $J$ are the \emph{projective degrees} of the rational map $\varphi_J$ \cite[Example 19.4]{Harris}. In particular, $e_k$ is the degree of $\varphi_J$ times the degree of the image of $\varphi_J$.
\end{remark}

The notion of mixed multiplicities can be extended to a sequence of ideals $J_1,\ldots,J_s$. Consider the standard $\mathbb{N}^{s+1}$-graded algebra
\[
R=R(\mathfrak{m}|J_1,\ldots,J_s) = \bigoplus_{(u,v_1,\ldots,v_s) \in \mathbb{N}^{s+1}} \mathfrak{m}^uJ_1^{v_1} \cdots J_s^{v_s}/\mathfrak{m}^{u+1}J_1^{v_1} \cdots J_s^{v_s}.
\]
The \emph{mixed multiplicities} $e_\mathbf{i}=e_{\mathbf{i}}(\mathfrak{m}|J_1,\ldots,J_s)$ are defined by the expression
\[
\text{HP}_{R}(u,v_1,\ldots,v_s) = \sum_{|\mathbf{i}|=n} \frac{e_\mathbf{i}}{i_0! i_1! \cdots i_s!} u^{i_0}v_1^{i_1}\cdots v_s^{i_s} + \text{(lower degree terms)}
\]
where the sum is over the sequences $\mathbf{i}=(i_0,i_1,\ldots,i_s)$ of nonnegative integers whose sum is $n$. 
Trung and Verma show in \cite[Corollary 1.6]{Trung-Verma} that positive mixed multiplicities can be expressed as Hilbert-Samuel multiplicities. 
Let $S$ be a local ring with infinite residue field and $J_1,\ldots,J_s$ be ideals of $S$. Suppose that $\text{HP}_R$ has the total degree $n$ and let $\mathbf{i} =(i_0,i_1,\ldots,i_s)$ be any sequence of nonnegative integers whose sum is $n$. 

\begin{theorem}[Trung-Verma]\label{T}
Let $Q$ be an ideal generated by $i_1$ general elements\footnote{We say that a property holds for a \emph{general element} $f$ of an ideal $(f_1,\ldots,f_m)$ in a local ring with infinite residue field $\kappa$ if there exists a nonempty Zariski-open subset $U \subseteq \kappa^m$ such that whenever $f=\sum_{k=1}^m c_k f_k$ and the image of $(c_1,\ldots,c_m)$ in $\kappa^m$ belongs to $U$, the property holds for $f$.} in $J_1$, \ldots, and $i_s$ general elements in $J_s$. Then 
\[
e_\mathbf{i}(\mathfrak{m}|J_1,\ldots,J_s)>0 \quad \text{iff} \quad \dim S/(Q:J_1 \cdots J_s^\infty)=i_0+1.
\]
In this case,
\[
e_\mathbf{i}(\mathfrak{m}|J_1,\ldots,J_s) = e\big(\mathfrak{m}, S/(Q:J_1\cdots J_s^\infty)\big).
\]
\end{theorem}

\noindent One readily verifies that the analogous statement holds for a standard graded ring $S$ over an infinite field, the irrelevant ideal $\mathfrak{m}$, and homogeneous ideals $J_1,\ldots,J_s$. We refer to \cite{Trung, Trung-Verma} for details and more general statements.

\subsection{Mixed multiplicities of ideals and mixed volumes of polytopes}

The $n$-dimensional volume $\text{V}_n$ of a nonnegative linear combination $v_1\Delta_1+\cdots+v_n\Delta_n$ of convex bodies in $\mathbb{R}^n$ is a homogeneous polynomial in the coefficients $v_1,\ldots,v_n$. The \emph{mixed volume} of $\Delta_1,\ldots,\Delta_n$ is defined to be the coefficient of the monomial $v_1v_2\cdots v_n$ in the homogeneous polynomial. We follow the convention of \cite[Chapter 7]{Cox-Little-Oshea} and write $\text{MV}_n$ for the mixed volume of convex polytopes in $\mathbb{R}^n$. For example,
\[
\text{MV}_n(\Delta,\ldots,\Delta)=1
\]
for the standard $n$-dimensional simplex $\Delta \subset \mathbb{R}^n$.
It follows from Ehrhart's theorem \cite[Section 6.3]{Bruns-Herzog} that the multiplicity of a toric algebra is the normalized volume of the associated lattice polytope. Trung and Verma use this relation to show in \cite[Corollary 2.5]{Trung-Verma} that mixed volumes of lattice convex polytopes in $\mathbb{R}^n$ are mixed multiplicities of certain monomial ideals.  Let $\mathbb{K}$ be a field.

\begin{theorem}[Trung-Verma]\label{mixed_volume}
Let $\Delta_1,\ldots,\Delta_n$ be lattice convex polytopes in $\mathbb{R}^n$. Let $J_i$ be an ideal of $\mathbb{K}[z_0,z_1,\ldots,z_n]$ generated by a set of monomials of the same degree such that $\Delta_i$ is the convex hull of exponents of their dehomogenized monomials in $\mathbb{K}[z_1,\ldots,z_n]$. Then
\[
\text{MV}_n(\Delta_1,\ldots,\Delta_n) = e_{(0,1,\ldots,1)}(\mathfrak{m}|J_1,\ldots,J_n).
\]
\end{theorem}

\noindent Therefore, by the Alexandrov-Fenchel inequality on mixed volumes of convex bodies \cite[Theorem 6.3.1]{Schneider}, if $J$ is an ideal generated by monomials of the same degree, then
the mixed multiplicities of $\mathfrak{m}$ and $J$ form a log-concave sequence. In this simple setting, a question of Trung and Verma can be stated as follows.

\begin{question}
Under what conditions on $J$ do the $e_i$ form a log-concave sequence?
\end{question}

\noindent In \cite[Question 2.7]{Trung-Verma}, Trung and Verma suggest the condition $\text{ht}(J) =n$. Corollary \ref{mixed-multiplicity} says that when $J$ is generated by elements of the same degree, the mixed multiplicities of $\mathfrak{m}$ and $J$ form a log-concave sequence. Example \ref{counter-example} shows that the answer to the question in its original formulation is no in general.

\begin{remark}\label{TRS}
The question is interesting in view of a theorem of Teissier \cite{Teissier2} and Rees-Sharp \cite{Rees-Sharp} on mixed multiplicities. The theorem says that if $J$ is an $\mathfrak{m}$-primary ideal in a local ring, then the mixed multiplicities of $\mathfrak{m}$ and $J$ form a \emph{log-convex} sequence. See \cite[Remark 1.6.8]{Lazarsfeld} for a Hodge-theoretic explanation.
\end{remark}

\section{The main results}

\subsection{Milnor numbers of projective hypersurfaces}

Let $h$ be a nonconstant homogeneous polynomial in $\mathbb{C}[z_0,\ldots,z_n]$ and $J_h$ be the Jacobian ideal of $h$. Denote
\begin{eqnarray*}
V(h) &=& \big\{p \in \mathbb{P}^n \mid h(p)=0 \big\},\\
D(h) &=& \big\{p \in \mathbb{P}^n \mid h(p) \neq 0\big\},
\end{eqnarray*}
where $\mathbb{P}^n$ is the $n$-dimensional complex projective space.

\begin{definition}\label{definition}
We define $\mu^i(h)$ to be the $i$-th mixed multiplicity of $\mathfrak{m}$ and $J_h$.
\end{definition}

Theorem \ref{A} relates the numbers $\mu^i(h)$ to the topology of $D(h)$ by repeatedly applying a theorem of Dimca-Papadima \cite[Theorem 1]{Dimca-Papadima}. In view of Conjecture \ref{Rota-Welsh}, the main technical point is Lemma \ref{chain}, which asserts that the process of taking derivatives and taking a general hyperplane section is compatible in an asymptotic sense. Let us fix a sufficiently general flag of linear subspaces
\[
\mathbb{P}^0 \subset \mathbb{P}^1 \subset \cdots \subset \mathbb{P}^{n-1} \subset \mathbb{P}^n.
\]
For $i=0,\ldots,n$, set $V(h)_i = V(h) \cap \mathbb{P}^i$ and $D(h)_i = D(h) \cap \mathbb{P}^i$.

\begin{theorem}\label{A}
For $i=0,\ldots,n$, the following hold. Read $D(h)_{-1} = \varnothing$.
\begin{enumerate}
\item $D(h)_{i}$ is homotopy equivalent to a CW-complex obtained from $D(h)_{i-1}$ by attaching $\mu^{i}(h)$ cells of dimension $i$. In particular,
\[
\mu^{i}(h) = (-1)^{i} \chi\big( D(h)_{i} \setminus D(h)_{i-1} \big).
\]
\item $V(h)_{i} \setminus V(h)_{i-1}$ is homotopy equivalent to a bouquet of $\mu^{i}(h)$ spheres of dimension $i-1$. In particular,
\[
\mu^{i}(h) = \tilde b_{i-1}\big( V(h)_{i} \setminus V(h)_{i-1} \big).
\]
\end{enumerate}
\end{theorem}

As a corollary, we obtain a formula for the topological Euler characteristic of the complement $D(h)$ in terms of mixed multiplicities:
\[
\chi\big(D(h)\big)=\sum_{i=0}^n (-1)^i \mu^i(h).
\]
The point of the above formula is that the numbers $\mu^i(h)$ are effectively calculable in most computer algebra systems. Since $\chi$ is additive on complex algebraic varieties \cite[p.141]{Fulton}, the formula provides a way of computing topological Euler characteristics of arbitrary complex projective varieties and affine varieties.

\begin{remark}\label{CSM-class}
$\mu^i(h)$ are the projective degrees of the Gauss map
\[
\text{grad}(h) : \mathbb{P}^n \dashrightarrow \mathbb{P}^n, \qquad p \longmapsto \Bigg(\frac{\partial h}{\partial z_0}(p):\cdots: \frac{\partial h}{\partial z_n}(p)\Bigg).
\]
See Remark \ref{projective degree} above. With this identification, a theorem of Aluffi \cite[Theorem 2.1]{Aluffi} says that the push-forward of the \emph{Chern-Schwartz-MacPherson class} \cite{MacPherson} of the hypersurface is given by the formula
\[
c_{SM}\big({\mathbf 1}_{D(h)}\big) = \sum_{i=0}^n (-1)^i \mu^i(h) H^i (1+H)^{n-i},
\]
where ${\mathbf 1}_{D(h)}$ is the characteristic function of the complement and $H^i$ is the class of a codimension $i$ linear subspace in the Chow ring $A_*(\mathbb{P}^n)$. Therefore Theorem \ref{A} provides an alternative explanation of the formula of \cite{Aluffi},
\[
\chi\big(D(h)\big) = \int c_{SM}\big({\mathbf 1}_{D(h)}\big) = \sum_{i=0}^n (-1)^i \mu^i(h),
\]
where it is first proposed as an effective way of computing $\chi$ of arbitrary complex projective and affine varieties. 
\end{remark}

\begin{example}\label{mu-one}
Let $h$ be a nonconstant homogeneous polynomial in $S=\mathbb{C}[z_0,\ldots,z_n]$. Write $h=\prod_{i=1}^k g_i^{m_i}$, where the $g_i$ are distinct irreducible factors of $h$ and $m_i \ge 1$. Let $\sqrt{h}$ be the radical $\prod_{i=1}^k g_i$ and $d$ be the degree of $\sqrt{h}$. Applying Theorem \ref{T}, we see that
\[
\mu^0(h) = e(\mathfrak{m} , S ) =1
\]
and, for sufficiently general constants $c_0,c_1,\ldots,c_n \in \mathbb{C}$,
\begin{eqnarray*}
\mu^1(h) &=& e\Big(\mathfrak{m}, S \big/ \sum_{j=0}^n c_j \sum_{i=1}^k m_i g_1^{m_1} \cdots g_i^{m_i-1} \cdots g_k^{m_k} \frac{\partial g_i}{\partial z_j} : J_h^\infty \Big) \\
&=& e\Big(\mathfrak{m}, S \big/ \sum_{j=0}^n c_j \sum_{i=1}^k m_i g_1 \cdots \hat g_i \cdots g_k \frac{\partial g_i}{\partial z_j} : J_h^\infty \Big) \\
&=& e\Big(\mathfrak{m}, S \big/ \sum_{j=0}^n c_j \sum_{i=1}^k m_i g_1 \cdots \hat g_i \cdots g_k \frac{\partial g_i}{\partial z_j} \Big)  \\
&=& d-1,
\end{eqnarray*}
where \ $\hat{}$ \ indicates an omission of the corresponding factor. This agrees with the fact that $D(h)_0$ is a point and $D(h)_1$ is homotopic to a bouquet of $d-1$ circles.
\end{example}

\begin{example}
Suppose $h \in \mathbb{C}[z_0,\ldots,z_n]$ is reduced of degree $d$ and $V(h)$ has only isolated singular points, say at $p_1,\ldots,p_m$. Since $J_h$ has height $n$, sufficiently general linear combinations $a_1,\ldots,a_n$ of its generators form a regular sequence. Therefore, for $0 \le i < n$,
\begin{eqnarray*}
\mu^i(h) &=& e\big(\mathfrak{m},S/(a_1,\ldots,a_i):J_h^\infty\big) \\
&=& e\big(\mathfrak{m}, S/(a_1,\ldots,a_i)\big) \\
&=& (d-1)^i.
\end{eqnarray*}
From the same formula, we see that $\mu^n(h)$ equals the sum of the degrees of the components of $\text{Proj}\big(S/(a_1,\ldots,a_n)\big)$ whose support is not contained in the singular locus of $V(h)$. Since the degree of a component of $\text{Proj}\big(S/(a_1,\ldots,a_n)\big)$ supported on $p_i$ is the Milnor number $\mu(h,p_i)$ at $p_i$, we have
\[
\mu^n(h) =(d-1)^n -\sum_{i=1}^m \mu(h,p_i).
\]
Compare \cite[Corollary 5.4.4]{Dimca}.
\end{example}

The numbers $\mu^i(h)$ satisfy a version of Kouchnirenko's theorem \cite{Kouchnirenko} relating the Milnor number with the Newton polytope. 

\begin{definition}
For any nonzero homogeneous $h \in \mathbb{C}[z_0,z_1,\ldots,z_n]$, we define $\Delta_h \subset \mathbb{R}^n$ to be the convex hull of exponents of dehomogenized monomials of $\mathbb{C}[z_1,\ldots,z_n]$ appearing in one of the partial derivatives of $h$. 
\end{definition}

\noindent Note that $\Delta_h$ is determined by the Newton polytope of $h$.

\begin{example}
Let $h$ be the degree $d$ homogeneous polynomial in $\mathbb{C}[z_0,z_1]$,
\[
h = \sum_{i=a}^b c_iz_0^{d-i}z_1^i 
= z_0^{d-b}z_1^a \Bigg( \sum_{i=a}^{b} c_{i}z_0^{b-i}z_1^{i-a}\Bigg),
\]
with $0 \le a \le b \le d$ and nonzero $c_a,c_b \in \mathbb{C}$. Then $\Delta_h$ is the closed interval
\[
\Delta_h = \left\{\begin{array}{ll}
\hspace{0.5mm} [a-1,b] & \text{if $a \neq 0$ and $b \neq d$},\\
\hspace{0.5mm} [a-1,b-1] & \text{if $a \neq 0$ and $b = d$},\\
\hspace{0.5mm} [a,b] & \text{if $a = 0$ and $b \neq d$},\\
\hspace{0.5mm} [a,b-1] & \text{if $a = 0$ and $b = d$}.
\end{array}\right.
\]
Note that the number of \emph{distinct} solutions $N$ of $h=0$ in $\mathbb{P}^1$ is at most
\[
N \le \left\{\begin{array}{ll}
\hspace{0.5mm} b-a+2 & \text{if $a \neq 0$ and $b \neq d$},\\
\hspace{0.5mm} b-a+1 & \text{if $a \neq 0$ and $b = d$},\\
\hspace{0.5mm} b-a+1 & \text{if $a = 0$ and $b \neq d$},\\
\hspace{0.5mm} b-a & \text{if $a = 0$ and $b = d$}
\end{array}\right.
\]
and that the equality is obtained in each case if the coefficients $c_i$ are chosen in a sufficiently general way. Since $D(h)$ is homotopic to the wedge of $N-1$ circles, the above inequalities may be written as
\[
b_1\big(D(h)\big) \le \text{MV}_1(\Delta_h).
\]
\end{example}

\begin{theorem}\label{bound}
Let $h$ be any homogeneous polynomial in $\mathbb{C}[z_0,\ldots,z_n]$. For $i=0,\ldots, n$, we have
\[
b_i\big(D(h)\big) \le \mu^i(h) \le \text{MV}_n(\underbrace{\Delta,\ldots,\Delta}_{n-i}, \underbrace{\Delta_h,\ldots,\Delta_h}_i),
\]
where $\Delta$ is the standard $n$-dimensional simplex in $\mathbb{R}^n$.
\end{theorem}

\noindent Therefore, if $h$ has a small Newton polytope under some choice of coordinates, then the Betti numbers of $D(h)$ cannot be large.

\begin{example}\label{sharp}
Let $h$ be the product of variables $z_0 z_1 \cdots z_n \in \mathbb{C}[z_0,\ldots,z_n]$. Then $D(h)$ is the complex torus $(\mathbb{C}^*)^n$, so the Betti numbers are the binomial coefficients $\binom{n}{i}$. 
We compare the Betti numbers with the mixed volumes of $\Delta$ and $\Delta_h$. Since $\Delta_h$ is a translation of $-\Delta$, we may replace $\Delta_h$ by $-\Delta$ when computing the mixed volumes. For $I \subseteq [n]=\{1,\ldots,n\}$, write $\mathbb{R}^n_I$ for the orthant
\[
\mathbb{R}^n_I = \bigcap_{p,q} \hspace{1mm} \big\{x_p \ge 0, x_q \le 0 \big\} ,
\]
where the intersection is over all $p \in [n] \setminus I$ and $q \in I$.
Then, for any nonnegative $a$ and $b$,
\begin{eqnarray*}
a\Delta+(-b\Delta) &=& \bigcup_{I \subseteq [n]} \Big( a\Delta+(-b\Delta) \Big) \cap \mathbb{R}^n_I \\
& =& \bigcup_{I \subseteq [n]} \text{conv}\big\{\mathbf{0}, a\mathbf{e}_p,-b\mathbf{e}_q,a\mathbf{e}_p-b\mathbf{e}_q\big\}_{p \in [n] \setminus I, q \in I} \\
&=& \bigcup_{I \subseteq [n]} \Big( a\Delta_{|[n] \setminus I|} \times b\Delta_{|I|} \Big),
\end{eqnarray*}
where $\mathbf{e}_i$ are the standard unit vectors of $\mathbb{R}^n$ and $\Delta_k$ is the standard $k$-dimensional simplex. Therefore we have
\[
\text{V}_n\big(a\Delta+(-b\Delta)\big) = \sum_{I \subseteq [n]} \frac{a^{|[n] \setminus I|}b^{|I|}}{|[n] \setminus I|!|I|!} = \sum_{i=0}^n {n \choose i} \frac{a^{n-i}b^{i}}{(n-i)!i!}
\]
and hence
\[
\text{MV}_n(\underbrace{\Delta,\ldots,\Delta}_{n-i}, \underbrace{\Delta_h,\ldots,\Delta_h}_i) = \binom{n}{i}.
\]
We note that equality holds throughout in Theorem \ref{bound}, for all $i=0,\ldots,n$ and any $n \ge 1$.
\end{example}

\begin{example}
There are polytopes in $\mathbb{R}^{n+1}$ such that the second inequality of Theorem \ref{bound} is strict for some $i$ for all $h$ having the given polytope as the Newton polytope. For example, consider the homogeneous polynomial in $\mathbb{C}[z_0,z_1,z_2]$,
\[
h=z_1(c z_0z_1-c' z_2^2),
\]
where $c,c' \in \mathbb{C}$ are any nonzero constants. Then the inequalities of Theorem \ref{bound} for $i=2$ read
\[
0 \le 1 \le 2.
\]
One can show that almost all homogeneous polynomials with a given Newton polytope share the numbers $\mu^i(h)$. An explicit formula for the numbers will appear elsewhere.
\end{example}

\subsection{Representable homology classes of $\mathbb{P}^n \times \mathbb{P}^m$}

An \emph{algebraic variety} is a reduced and irreducible scheme of finite type over an algebraically closed field. Given an algebraic variety $X$, we pose the following question: which homology classes of $X$ can be represented by a subvariety? See \cite[Question 1.3]{Hartshorne} for a related discussion.

\begin{definition}
We say that $\xi \in A_*(X)$ is \emph{representable} if there is a subvariety $Z$ of $X$ with $\xi=[Z]$.
\end{definition}

Theorem \ref{B} asserts that representable homology classes of $\mathbb{P}^n \times \mathbb{P}^m$ correspond to log-concave sequences of nonnegative numbers with no internal zero. We start by giving two examples illustrating cases of exceptional nature.
 
\begin{example}
Let $X$ be the projective space $\mathbb{P}^n$. Write $\xi \in A_k(\mathbb{P}^n)$ as a multiple
\[
\xi=e \hspace{0.5mm} \big[\mathbb{P}^k \big]
\]
for some $e \in \mathbb{Z}$.
\begin{enumerate}[1.]
\item If $k=0$ or $k=n$, then $\xi$ is representable iff $e=1$.
\item If otherwise, then $\xi$ is representable iff $e \ge 1$.
\end{enumerate}
In the latter case, one may use an irreducible hypersurface of degree $e$ in $\mathbb{P}^{k+1} \subseteq \mathbb{P}^n$ to represent $\xi$.
\end{example}

\begin{example}
Let $X$ be the surface $\mathbb{P}^1 \times \mathbb{P}^1$. Write $\xi \in A_1(\mathbb{P}^1 \times \mathbb{P}^1)$ as a linear combination
\[
\xi=e_0 \big[\mathbb{P}^1 \times \mathbb{P}^0 \big]+e_1 \big[ \mathbb{P}^0 \times \mathbb{P}^1\big]
\]
for some $e_0,e_1 \in \mathbb{Z}$. 
\begin{enumerate}[1.]
\item If one of the $e_i$ is zero, then $\xi$ is representable iff the other $e_i$ is $1$.
\item If otherwise, then $\xi$ is representable iff both $e_i$ are positive.
\end{enumerate}
Note that in the former case, by the fundamental theorem of algebra, a nonconstant bihomogeneous polynomial of degree $(e_0,e_1)$ is reducible if one of the $e_i$ is zero, unless the other $e_i$ is $1$.
\end{example}

We characterize representable homology classes of $\mathbb{P}^n \times \mathbb{P}^m$ up to a positive integer multiple. 

\begin{theorem}\label{B}
Write $\xi \in A_k(\mathbb{P}^n \times \mathbb{P}^m)$ as an integral linear combination
\[
\xi =\sum_{i} e_i \big[\mathbb{P}^{k-i} \times \mathbb{P}^{i}\big],
\]
where the term containing $e_i$ is zero if $n<k-i$ or $m<i$. 
\begin{enumerate}[1.]
\item If $\xi$ is an integer multiple of either
\[
\big[\mathbb{P}^n \times \mathbb{P}^m\big], \big[\mathbb{P}^n \times \mathbb{P}^0\big], \big[\mathbb{P}^0 \times \mathbb{P}^m\big], \big[\mathbb{P}^0 \times \mathbb{P}^0\big],
\]
then $\xi$ is representable iff the integer is $1$.
\item If otherwise, some positive integer multiple of $\xi$ is representable iff the $e_i$ form a nonzero log-concave sequence of
nonnegative integers with no internal zeros.
\end{enumerate}
\end{theorem}


\noindent Combined with Remark \ref{projective degree}, Theorem \ref{B} implies the following.

\begin{corollary}\label{mixed-multiplicity}
If $J$ is an ideal of a standard graded domain over an algebraically closed field generated by elements of the same degree, then the mixed multiplicities of $\mathfrak{m}$ and $J$ form a log-concave sequence of nonnegative integers with no internal zeros.
\end{corollary}

\begin{example}\label{counter-example}
We show by example that the anwer to the following question of Trung and Verma \cite[Question 2.7]{Trung-Verma} is no in general.

\vspace{3mm}

\noindent \emph{Let $A$ be a local ring of dimension $n+1 \ge 3$, $J_1,\ldots,J_n$ be ideals of height $n$, and $\mathbf{i}=(0,1,\ldots,1)$. Is it true that
\[
e_{\mathbf{i}}(\mathfrak{m}|J_1,J_1,J_3,\ldots,J_n) e_{\mathbf{i}}(\mathfrak{m}|J_2,J_2,J_3,\ldots,J_n) \le e_{\mathbf{i}}(\mathfrak{m}|J_1,J_2,J_3,\ldots,J_n)^2?
\]}

\vspace{3mm}

\noindent Let $A$ be the power series ring $\mathbb{C}\{x,y,z\}$, $J_1 = (xy^2,y^3z,xz)$ and $J_2=(xy^2,y^3z,xz^2)$ ideals of $A$. Then $\text{ht}(J_1)=\text{ht}(J_2)=2$. However, using Theorem \ref{T} one computes
\begin{eqnarray*}
e_{(0,1,1)}(\mathfrak{m} | J_1,J_1) &=& 1, \\
e_{(0,1,1)}(\mathfrak{m} | J_1,J_2) &=& 1, \\
e_{(0,1,1)}(\mathfrak{m} | J_2,J_2 ) &=& 2 .
\end{eqnarray*}
More precisely, writing $\square$ for sufficiently general nonzero constants in $\mathbb{C}$,
\begin{eqnarray*}
e_{(0,1,1)}(\mathfrak{m} | J_1,J_1) &=& e\big(\mathfrak{m}, A/ (\square xy^2+\square y^3z +\square xz,\square xy^2+\square y^3z +\square xz) : J_1^\infty \big)\\
&=& e\big(\mathfrak{m}, A/ (\square xy^2 +\square xz,\square y^3z +\square xz) : J_1^\infty \big)\\
&=& e\big(\mathfrak{m}, A/ (\square y^2 +\square z,\square y^3 +\square x) : J_1^\infty \big)\\
&=& e\big(\mathfrak{m}, A/ (\square y^2 +\square z,\square y^3 +\square x) \big)\\
&=& 1,
\end{eqnarray*}
\begin{eqnarray*}
e_{(0,1,1)}(\mathfrak{m} | J_1,J_2) &=& e\big(\mathfrak{m}, A/ (\square xy^2+\square y^3z +\square xz,\square xy^2+\square y^3z +\square xz^2) : J_1J_2^\infty \big)\\
&=& e\big(\mathfrak{m}, A/ (\square xy^2+\square xz+\square xz^2,\square y^3z +\square xz^2+\square xz) : J_1J_2^\infty \big)\\
&=& e\big(\mathfrak{m}, A/ (\square y^2+\square z+\square z^2,\square y^3 +\square xz+\square x) : J_1J_2^\infty \big)\\
&=& e\big(\mathfrak{m}, A/ (\square y^2+\square z+\square z^2,\square y^3 +\square xz+\square x) \big)\\
&=& 1,
\end{eqnarray*}
\begin{eqnarray*}
e_{(0,1,1)}(\mathfrak{m} | J_2,J_2) &=& e\big(\mathfrak{m}, A/ (\square xy^2+\square y^3z +\square xz^2,\square xy^2+\square y^3z +\square xz^2) : J_2^\infty \big)\\
&=& e\big(\mathfrak{m}, A/ (\square xy^2+\square xz^2,\square y^3z +\square xz^2) : J_2^\infty \big)\\
&=& e\big(\mathfrak{m}, A/ (\square y^2+\square z^2,\square y^3 +\square xz) : J_2^\infty \big)\\
&=& e\big(\mathfrak{m}, A/ (\square y^2+\square z^2,\square yz^2 +\square xz) : J_2^\infty \big)\\
&=& e\big(\mathfrak{m}, A/ (\square y^2+\square z^2,\square yz+\square x) : J_2^\infty \big)\\
&=& e\big(\mathfrak{m}, A/ (\square y^2+\square z^2,\square yz+\square x) \big)\\
&=& 2.
\end{eqnarray*}
In each of the three computations above, the first equality is an application of Theorem \ref{T}, the second is a Gaussian elimination, and the third is a result of saturation. The same technique is used twice in the computation of $e_{(0,1,1)}(\mathfrak{m} | J_2,J_2)$.
\end{example}

\subsection{Log-concavity of characteristic polynomials}\label{Graphic-Arrangement}

Suppose $h \in \mathbb{C}[z_0,\ldots,z_n]$ is a product of linear forms.  Let $\widetilde{\mathcal{A}} \subset \mathbb{C}^{n+1}$ be the central hyperplane arrangement defined by $h$ and $\mathcal{A} \subset \mathbb{P}^n$ be the corresponding projective arrangement. 

\begin{definition}
Let $H$ be a hyperplane in $\mathcal{A}$. The \emph{decone} of $\mathcal{A}$ is an affine arrangement
\[
\overline{\mathcal{A}} = \overline{\mathcal{A}}^H \subset \mathbb{C}^n
\]
obtained from $\mathcal{A}$ by declaring $H$ to be the hyperplane at infinity.
\end{definition}

\noindent The lattice of flats $\mathscr{L}_{\overline{A}}$ is isomorphic to the sublattice of $\mathscr{L}_{\widetilde{A}}$ consisting of all the flats not contained in the hyperplane $H$. It follows from the modular element factorization \cite[Corollary 4.8]{Stanley3} that
\[
\chi_{\overline{\mathcal{A}}}(q) = \chi_{\widetilde{\mathcal{A}}}(q)/(q-1).
\]

\begin{corollary}\label{formula}
We have
\[
\chi_{\overline{\mathcal{A}}}(q)= \sum_{i=0}^n (-1)^i b_i\big(D(h)\big) q^{n-i}  = \sum_{i=0}^n (-1)^i \mu^i(h) q^{n-i}. 
\]
\end{corollary}

\begin{proof}
The first equality is a theorem of Orlik and Solomon applied to the affine arrangement $\overline{A} \subset \mathbb{C}^n$ \cite{Orlik-Solomon}. We adapt an argument of Randell \cite{Randell} to prove the second equality. Fix a sufficiently general flag of linear subspaces
\[
\mathbb{P}^0 \subset \mathbb{P}^1 \subset \cdots \subset \mathbb{P}^{n-1} \subset \mathbb{P}^n
\]
and a nonnegative integer $k < n$. If the linear subspace $\mathbb{P}^k$ is transversal to all the flats of $\mathcal{A}$ of relevant dimensions, then the lattice of flats $\mathscr{L}_{\overline{\mathcal{A}|_{\mathbb{P}^k}}}$ is isomorphic to the sublattice of $\mathscr{L}_{\overline{\mathcal{A}}}$ consisting of all the flats of codimension $\le k$. It follows that $\chi_{\overline{\mathcal{A}|_{\mathbb{P}^k}}}(q)$ is a truncation of $\chi_{\overline{\mathcal{A}}}(q)$. Combined with the first equality, we have
\[
\chi_{\overline{\mathcal{A}|_{\mathbb{P}^k}}}(q)  = \sum_{i=0}^{k} (-1)^i b_i\big(D(h)\big) q^{k-i}.
\]
In particular, $b_k\big(D(h) \cap \mathbb{P}^k\big) = b_k\big( D(h)\big)$. If furthermore $\mathbb{P}^k$ is chosen so that it satisfies the genericity assumption of Theorem \ref{A}, then $D(h) \cap \mathbb{P}^{k+1}$ is obtained from $D(h) \cap \mathbb{P}^k$ by attaching $\mu^{k+1}(h)$ cells of dimension $k+1$. Since the attaching does not alter the $k$-th Betti number, this attaching map should be homologically trivial. Therefore 
\[
b_{k+1}\big( D(h)\big) = b_{k+1}\big(D(h) \cap \mathbb{P}^{k+1}\big) = \mu^{k+1}(h).
\]
\end{proof}

\begin{remark}
Using the additivity of the Chern-Schwartz-MacPherson class \cite{MacPherson}, it is possible to prove the equality
\[
\chi_{\overline{\mathcal{A}}}(q)=  \sum_{i=0}^n (-1)^i \mu^i(h) q^{n-i}
\]
without using Theorem \ref{A} nor the theorem of Orlik and Solomon. Example \ref{mu-one} shows that the equality holds when $n=1$. 
Therefore, by induction on the dimension, it suffices to show that both polynomials satisfy the same recursive formula for a \emph{triple} of affine arrangements $(\overline{\mathcal{A}},\overline{\mathcal{A}}',\overline{\mathcal{A}}'')$ \cite[Definition 1.14]{Orlik-Terao}. Let $h,h',h''$ be homogeneous polynomials corresponding to the triple. Then by the additivity of the Chern-Schwartz-MacPherson class,
\begin{eqnarray*}
c_{SM}\big(\mathbf{1}_{D(h)}\big) &=& c_{SM}\big(\mathbf{1}_{D(h')} - \iota_*\mathbf{1}_{D(h'')}\big)\\
&=& c_{SM}\big(\mathbf{1}_{D(h')}\big) - c_{SM}\big(\iota_*\mathbf{1}_{D(h'')}\big)\\
&=& c_{SM}\big(\mathbf{1}_{D(h')}\big) - \iota_* c_{SM}\big(\mathbf{1}_{D(h'')}\big),
\end{eqnarray*}
where $\mathbf{1}$ stands for the characteristic function and $\iota$ is the inclusion of the distinguished hyperplane into $\mathbb{P}^n$. Using the formula of Remark \ref{CSM-class}, this implies that
\[
\mu^i(h) = \mu^i(h')+\mu^{i-1}(h'')
\]
for $0< i \le n$, which exactly corresponds to the inductive formula of Brylawski and Zaslavsky for the characteristic polynomial of triples \cite[Theorem 2.56]{Orlik-Terao}.
\end{remark}

\begin{corollary}\label{main}
If $M$ is representable over a field of characteristic zero, then the coefficients of $\chi_M(q)$ form a sign-alternating log-concave sequence of integers with no internal zeros.
\end{corollary}

\begin{proof}
Suppose $M$ is representable over $\mathbb{C}$. Let $\mathcal{A} \subset \mathbb{P}^n$ and $\widetilde{\mathcal{A}} \subset \mathbb{C}^{n+1}$ be the projective and the central arrangements, respectively, representing $M$. If $\overline{\mathcal{A}}$ is a decone of $\mathcal{A}$, then
\[
\chi_M(q) = \chi_{\widetilde{\mathcal{A}}}(q) = (q-1) \chi_{\overline{\mathcal{A}}}(q).
\]
Corollary \ref{mixed-multiplicity} together with Corollary \ref{formula} says that the absolute values of the coefficients of $ \chi_{\overline{\mathcal{A}}}(q)$ form a log-concave sequence of nonnegative integers with no internal zeros, and hence the same for $\chi_M(q)$. (In general, the convolution product of two log-concave sequences is again log-concave. It is easy to check this directly in our case.) This shows that the assertion holds for matroids representable over $\mathbb{C}$. 

We claim that simple matroids representable over a field of characteristic zero are in fact representable over $\mathbb{C}$. For this we check that matroid representability can be expressed in a first-order sentence in the language of fields and appeal to the completeness of $\text{ACF}_0$ \cite[Corollary 3.2.3]{Marker}. If $M$ is a simple matroid of rank $r$ on a set $E$ of cardinality $n$, then $M$ is representable over a field $\mathbb{K}$ if and only if the following formula is valid over $\mathbb{K}$ \cite[Section 9.1]{Welsh}: There are $n$ column vectors of length $r$ labelled by the elements of $E$, where a subset of $E$ is independent if and only if the corresponding set of vectors is linearly independent.

\end{proof}

\section{Milnor numbers of projective hypersurfaces}

\subsection{Proof of Theorem \ref{A}}

Let $h$ be a nonconstant homogeneous polynomial in $\mathbb{C}[z_0,\ldots,z_n]$ and $J_h$ be the Jacobian ideal of $h$. Associated to $h$ is the gradient map
\[
\text{grad}(h) : \mathbb{P}^n \dashrightarrow \mathbb{P}^n, \qquad p \longmapsto \Bigg(\frac{\partial h}{\partial z_0}(p):\cdots: \frac{\partial h}{\partial z_n}(p)\Bigg).
\]
We write $\deg(h)$ to denote the degree of the rational map $\text{grad}(h)$.  
Theorem \ref{A} depends on a theorem of Dimca and Papadima \cite[Theorem 1]{Dimca-Papadima} expressing $\deg(h)$ as the number of $n$-cells that have to be added to obtain a hypersurface complement from its general hyperplane section.  

\begin{theorem}[Dimca-Papadima]\label{DP}
Let $H \subset \mathbb{P}^n$ be a general hyperplane.
\begin{enumerate}[1.]
\item $D(h)$ is homotopy equivalent to a CW complex obtained from $D(h) \cap H$ by attaching $\deg(h)$ cells of dimension $n$.
\item $V(h) \setminus H$ is homotopic to a bouquet of $\deg(h)$ spheres of dimension $n-1$.
\end{enumerate}
\end{theorem}

\noindent In particular, $\deg(h)$ depends only on the set $V(h)$. 
Our goal is to identify the mixed multiplicity $\mu^i(h)$ with the degree of the gradient map of a general $i$-dimensional section of $V(h)$. 

\begin{lemma}\label{Lefschetz}
Let $J$ be a homogeneous ideal of a standard graded algebra $S$ over a field. Then, for a sufficiently general linear form $x$ in $S$, $\overline{S}=S/xS$,
\[
\text{HP}_{R(\mathfrak{m} \overline{S}, J \overline{S})}(u,v) = \text{HP}_{R(\mathfrak{m} | J)}(u,v) - \text{HP}_{R(\mathfrak{m} | J)}(u-1,v).
\]
It follows that
\[
e_i(\mathfrak{m}\overline{S}| J \overline{S}) = e_i(\mathfrak{m} | J) \quad \text{for $0 \le i < \deg \text{HP}_{R(\mathfrak{m} | J)}$}.
\]
\end{lemma}

\begin{proof}
See \cite[Lemma 1.3]{Trung-Verma}.
\end{proof}

Recall that a subideal $I \subseteq J$ is said to be a \emph{reduction} of $J$ if there exists a nonnegative integer $k$ such that $J^{k+1}=IJ^k$. If $J$ is finitely generated, then $I$ is a reduction of $J$ if and only if the integral closures of $I$ and $J$ coincide \cite[Corollary 1.2.5]{Swanson-Huneke}.
We refer to \cite{Swanson-Huneke} for general background on the reduction and the integral closure of ideals.

\begin{lemma}\label{reduction}
If $I$ is a reduction of $J$, then 
\[
e_i(\mathfrak{m} | I) = e_i(\mathfrak{m} | J) \quad \text{for $0 \le i \le \deg \text{HP}_{R(\mathfrak{m} | J)}$}.
\]
\end{lemma}

\begin{proof}
See \cite[Corollary 3.8]{Trung}.
\end{proof}

The following lemma is a version of Teissier's idealistic Bertini theorem on families of singular complex spaces \cite[Section 2.2]{Teissier4}. We give a simple proof in our simple setting. For more results of this type, and for the proof of the general case, we refer the reader to \cite{Gaffney1,Gaffney2,Teissier,Teissier4}. Let $h$ be a nonconstant homogeneous polynomial in $S=\mathbb{C}[z_0,\ldots,z_n]$.

\begin{lemma}\label{chain}
Let $x$ be a nonzero linear form in $S$, $\overline{S}=S/xS$, and $J_{\overline{h}}$ be the Jacobian ideal of the class of $h$ in $\overline{S}$. Then, for a sufficiently general $x$, $J_{\overline{h}}$ is a reduction of $J_h \overline{S}$.
\end{lemma}

\begin{proof}
It suffices to prove when the partial derivatives of $h$ are linearly independent. Let $V$ be the vector space of linear forms in $S$ and let $W$ be the vector space spanned by the partial derivatives of $h$. If $x$ is a linear form $c_0z_0+\cdots+c_nz_n$ with $c_n \neq 0$, then $\overline{S}$ is the polynomial ring generated by the classes of $z_0,\ldots,z_{n-1}$. By the chain rule, $J_{\overline{h}}$ is generated by the restrictions of the polynomials
\[
c_n \frac{\partial h}{\partial z_i} - c_i \frac{\partial h}{\partial z_n}, \qquad 0 \le i < n.
\]
This identifies an affine piece of the projective space of lines in $V$ with an affine piece of the projective space of hyperplanes in $W$.
We claim that the image in $\overline{S}$ of a general hyperplane in $W$ generates a reduction of $J_h \overline{S}$.

In general, let $R = \bigoplus_{k \in \mathbb{N}} R_k$ be a standard graded ring of dimension $n$ over an infinite field and $J$ be an ideal generated by a subspace $L$ of dimension $\ge n$ in some $R_k$. Then a sufficiently general subspace of dimension $n$ in $L$ generates a reduction $I$ of $J$. To see this, consider the graded map between the fiber rings
\[
\mathcal{F}_I = \bigoplus_{k \in \mathbb{N}} \frac{I^k}{\mathfrak{m}I^k} \longrightarrow \mathcal{F}_J = \bigoplus_{k \in \mathbb{N}} \frac{J^k}{\mathfrak{m}J^k}.
\]
One shows that $J$ is integral over $I$ in the ideal-theoretic sense iff $\mathcal{F}_J$ is integral over $\mathcal{F}_I$ in the ring-theoretic sense \cite[Proposition 8.2.4]{Swanson-Huneke}. The conclusion follows from the graded Noether normalization theorem applied to $\mathcal{F}_J$ because the dimension of $\mathcal{F}_J$ is at most $n$ \cite[Proposition 5.1.6]{Swanson-Huneke}. This technique is due to Samuel \cite{Samuel} and Northcott-Rees \cite{Northcott-Rees}.
\end{proof}

\begin{proof}[Proof of Theorem \ref{A}]
We induct on the dimension $n$. The case $n=1$ is dealt with in Example \ref{mu-one}. For general $n$, let $x \in \mathfrak{m}$ be a sufficiently general linear form, $\overline{S}=S/xS$, and $\overline{h}$ be the image of $h$ in $\overline{S}$. The induction hypothesis applied to $\overline{h}$ says that the assertions for $i<n$ hold with $\mu^i(\overline{h})$ in place of $\mu^i(h)$. However, for $i<n$, we have
\[
\begin{array}{llll}\vspace{2mm}
\mu^i(\overline{h}) &=& e_i\big(\mathfrak{m}\overline{S} | J_{\overline{h}}\big) & \hspace{15mm} \text{by definition}, \\ \vspace{2mm}
&=& e_i\big(\mathfrak{m}\overline{S} | J_h \overline{S}\big) & \hspace{15mm} \text{by Lemma \ref{reduction} and Lemma \ref{chain}}, \\ \vspace{2mm}
&=& e_i\big(\mathfrak{m}|J_h\big) & \hspace{15mm} \text{by Lemma \ref{Lefschetz}},\\
&=& \mu^i(h) & \hspace{15mm} \text{by definition}. 
\end{array}
\]
This proves the assertions for $i<n$. For $i=n$, we recall from Remark \ref{projective degree} that $\mu^n(h)$ is the $n$-th projective degree of $\text{grad}(h)$. Since the target of $\text{grad}(h)$ is $\mathbb{P}^n$, the $n$-th projective degree equals $\deg(h)$, the degree of the rational map $\text{grad}(h)$. With this identification, Theorem \ref{DP} says that the assertions hold for $i=n$.
\end{proof}

\subsection{Proof of Theorem \ref{bound}}

\begin{lemma}\label{monoton}
Let $S$ be a standard graded ring of dimension $n+1$ over a field, $\mathfrak{m}$ the irrelevant ideal of $S$.
\begin{enumerate}[1.]
\item Suppose $I$ and $J$ are ideals of positive height, both generated by elements of the same degree $r \ge 0$. If $I \subseteq J$, then
\[
e_i(\mathfrak{m} | I) \le e_i(\mathfrak{m} | J) \qquad \text{for $i=0,\ldots,n$}.
\]
\item Suppose $I$ and $J$ are $\mathfrak{m}$-primary ideals. If $I \subseteq J$, then
\[
e_i(\mathfrak{m} | I) \ge e_i(\mathfrak{m} | J) \qquad \text{for $i=0,\ldots,n$}.
\]
\end{enumerate}
\end{lemma}

\begin{proof}
If $I$ and $J$ are ideals of positive height, then $\deg \text{HP}_{R(\mathfrak{m}|I)} = \deg \text{HP}_{R(\mathfrak{m}|J)} =n$ \cite[Theorem 1.2]{Trung-Verma}. We prove the first part by induction on $n$. Since $I$ and $J$ are generated by elements of the same degree,
\[
\dim_{S/\mathfrak{m}} \frac{\mathfrak{m}^u I^v}{\mathfrak{m}^{u+1} I^v} \le \dim_{S/\mathfrak{m}} \frac{\mathfrak{m}^u J^v}{\mathfrak{m}^{u+1} J^v} \qquad \text{for $u, v \ge 0$}
\]
because the former vector space is contained in the latter. 
If $n=1$, then the above condition implies that
\[
e_0(\mathfrak{m} | I)u+e_1(\mathfrak{m} | I)v \le e_0(\mathfrak{m} | J)u+e_1(\mathfrak{m} | J)v \quad \text{for all large $u$ and large $v$}
\]
and hence 
\[
e_i(\mathfrak{m}|I) \le e_i(\mathfrak{m} |J) \quad \text{for $i=0,1$}.
\]
Now suppose $n>1$ and assume the induction hypothesis. We know that
\[
\sum_{i=0}^n \frac{e_i(\mathfrak{m}|I)}{(n-i)! i!} u^{n-i} v^{i} \le \sum_{i=0}^n \frac{e_i(\mathfrak{m}|J)}{(n-i)! i!} u^{n-i} v^{i} \quad \text{for all large $u$ and large $v$}.
\]
Taking the limit $u/v \to 0$ while keeping $u$ and $v$ sufficiently large, we have
\[
e_n(\mathfrak{m} | I) \le e_n(\mathfrak{m} | J).
\]
For the remaining cases, choose a sufficiently general linear form $x$ in  $S$. Using Lemma \ref{Lefschetz} and the induction hypothesis for the triple $\overline{S}=S/xS$, $I\overline{S}$, $J\overline{S}$, we have
\[
e_i(\mathfrak{m}|I) = e_i(\mathfrak{m}\overline{S}|I\overline{S}) \le e_i(\mathfrak{m}\overline{S}|J\overline{S})  = e_i(\mathfrak{m}|I) \quad \text{for $i=0,\ldots,n-1$}.
\]

The second part can be proved in the same way except that one starts from
\[
\dim_{S/\mathfrak{m}} S/\mathfrak{m}^{u} I^v \ge \dim_{S/\mathfrak{m}} S/\mathfrak{m}^{u} J^v \qquad \text{for $u, v \ge 0$}.
\]
\end{proof}

\begin{proof}[Proof of Theorem \ref{bound}]
The first inequality is clear in view of Theorem \ref{A}. For the second inequality, note that $J_h$ is contained in the ideal generated by all the monomials appearing in one of the partial derivatives of $h$. Let us denote this latter ideal by $K_h$. By Lemma \ref{monoton}, $e_i(\mathfrak{m}|J_h) \le e_i(\mathfrak{m}|K_h)$ holds for $0 \le i \le n$. Computing $e_i(\mathfrak{m}|K_h)$ from the function
\[
\dim_\mathbb{C} \mathfrak{m}^{p} K_h^{q}/\mathfrak{m}^{p + 1} K_h^{q}, \quad p=u+\sum_{j=1}^{n-i} v_j, \quad q=\sum_{k=1}^i v_{n-i+k}
\]
shows that
\[
e_i(\mathfrak{m}|K_h) = e_{(0,1,\ldots,1)}(\mathfrak{m}|\underbrace{\mathfrak{m},\ldots,\mathfrak{m}}_{n-i},\underbrace{K_h,\ldots,K_h}_{i}) \quad \text{for $0 \le i \le n$}. 
\]
Using Theorem \ref{mixed_volume} to identify the right-hand side with the mixed volume of interest, we have
\[
\mu^i(h) \le \text{MV}_n(\underbrace{\Delta,\ldots,\Delta}_{n-i}, \underbrace{\Delta_h,\ldots,\Delta_h}_i)  \quad \text{for $0 \le i \le n$}. 
\]
\end{proof}

\section{Representable homology classes of $\mathbb{P}^n \times \mathbb{P}^m$}\label{LastSection}

\subsection{Proof of Theorem \ref{B}}

The main ingredient of Theorem \ref{B} is the Hodge-Teissier-Khovanskii inequality \cite{Khovanskii,Teissier3}. See also the presentations \cite{Gromov} and \cite[Section 1.6]{Lazarsfeld}. We introduce another proof using Okounkov bodies, blind to the characteristic and to the singularities, which is more akin to the convex geometric viewpoint of this paper. The proof closely follows the argument of \cite{Kaveh-Khovanskii}. First we recall the necessary facts on Okounkov bodies from \cite{Lazarsfeld-Mustata}. Let $D$ be a divisor on an $n$-dimensional algebraic variety $X$. The \emph{Okounkov body} of $D$, denoted $\Delta(D)$, is a compact convex subset of $\mathbb{R}^n$ with the following properties.

\begin{enumerate}[A.]
\item If $H$ is an ample\footnote{The first equality holds more generally for big divisors \cite[Theorem 2.3]{Lazarsfeld-Mustata}.} divisor on $X$, then
\[
n! \hspace{1mm} \text{V}_n\big(\Delta(H)\big) = \lim_{k \to \infty} \frac{h^0\big( X,\mathcal{O}_X(kH)\big)}{k^n/ n!} = \int_X \underbrace{H \cdot H \cdot \ldots \cdot H}_n.
\]
\item For any divisors $D_1$ and $D_2$ on $X$,
\[
\Delta(D_1) + \Delta(D_2) \subseteq \Delta(D_1+D_2).
\]
\end{enumerate}

\begin{lemma}[Teissier-Khovanskii]\label{TK}
Let $H_1,\ldots,H_n$ be nef divisors on an $n$-dimensional variety $X$. Then
\[
\Big(\int_X H_1 \cdot H_1 \cdot H_3 \cdot \ldots \cdot H_n \Big) \Big(\int_X H_2 \cdot H_2 \cdot H_3 \cdot \ldots \cdot H_n\Big) \le \Big(\int_X H_1 \cdot H_2 \cdot H_3 \cdot \ldots \cdot H_n \Big)^2.
\]
\end{lemma}

\begin{proof}
The proof is by induction on $n$. By Kleiman's theorem \cite[Theorem 1.4.23]{Lazarsfeld}, $H_i$ are the limits of rational ample divisor classes in the N\'eron-Severi space of $X$. Therefore it suffices to prove the inequality for rational ample divisor classes. Because of the homogeneity of the stated inequality, we may further assume that the $H_i$ are very ample integral divisors. If $n \ge 3$, then Bertini's theorem \cite[Corollary 6.11]{Jouanolou} allows us to apply the inductive hypothesis to $H_n$, which is a subvariety of $X$. Therefore we are reduced to the case of surfaces.

When $X$ is a surface and $H_1,H_2$ are ample divisors on $X$, by the Brunn-Minkowski inequality \cite[Theorem 6.1.1]{Schneider} and Property B above, we have
\[
\text{V}_2\big(\Delta(H_1)\big)^{1/2}+\text{V}_2\big(\Delta(H_2)\big)^{1/2} \le \text{V}_2\big(\Delta(H_1)+\Delta(H_2)\big)^{1/2} \le \text{V}_2\big(\Delta(H_1+H_2)\big)^{1/2}.
\]
Now Property A says that the square of the above inequality simplifies to
\[
\Big(\int_X H_1 \cdot H_1 \Big) \Big(\int_X H_2 \cdot H_2\Big) \le \Big(\int_X H_1 \cdot H_2 \Big)^2.
\]
\end{proof}

\begin{lemma}\label{no internal zeros}
Let $\mathscr{C} \subset \mathbb{R}^{n+1}$ be the set of all log-concave sequences of positive real numbers, and let $\mathscr{C}' \subset \mathbb{R}^{n+1}$ be the set of all log-concave sequences of nonnegative real numbers with no internal zeros. Then $\overline{\mathscr{C}} = \mathscr{C}'$.
\end{lemma}

\begin{proof}
We first show that a sequence $(e_0,e_1,\ldots,e_n) \in \mathscr{C}$ satisfies
\[
e_i^{k-j} e_k^{j-i} \le e_j^{k-i} \quad \text{for $0 \le i < j < k \le n$}.
\]
This can be shown by induction on $k-i$. 
By the induction hypothesis, we have
\[
e_i^{k-j-1}e_{k-1}^{j-i} \le e_j^{k-i-1} \quad \text{and} \quad
e_j e_k^{k-j-1} \le e_{k-1}^{k-j}.
\]
Eliminating $e_{k-1}$ in the first inequality using the second inequality, and using the assumption that $e_j$ is positive, we have what we want. The proved inequality shows that
\[
\overline{\mathscr{C}} \subseteq \mathscr{C}'.
\]
The reverse inclusion is easy to see. A log-concave sequence of nonnegative numbers with no internal zeros can be written as
\[
\ldots,0,0,0,a_1,a_2,\ldots,a_k,0,0,0,\ldots
\]
for some positive numbers $a_i$. This sequence is the limit of sequences in $\mathscr{C}$ of the form
\[
\ldots,\epsilon^5,\epsilon^3,\epsilon,a_1,a_2,\ldots,a_k,\epsilon,\epsilon^3,\epsilon^5,\ldots
\]
for sufficiently small positive $\epsilon$.
\end{proof}

We deal with the four exceptional cases in a separate lemma.

\begin{lemma}\label{exceptional}
If $\xi \in A_*(\mathbb{P}^n \times \mathbb{P}^m)$ is an integer multiple of either
$[\mathbb{P}^n \times \mathbb{P}^m], [\mathbb{P}^n \times \mathbb{P}^0], [\mathbb{P}^0 \times \mathbb{P}^m]$, or $[\mathbb{P}^0 \times \mathbb{P}^0]$,
then $\xi$ is representable iff the integer is $1$.
\end{lemma}

\begin{proof}
The interesting part is to prove the necessity when $\xi$ is an integer multiple of $[\mathbb{P}^n \times \mathbb{P}^0]$ or $[\mathbb{P}^0 \times \mathbb{P}^m]$. It is enough to consider the first case. If $\xi$ is represented by a subvariety $Z$, then a general hypersurface $H$ of the form $\mathbb{P}^n \times \mathbb{P}^{m-1} \subset \mathbb{P}^n \times \mathbb{P}^{m}$ is disjoint from $Z$, since otherwise the intersection $H\cap Z$ defines a nonzero class in $A_{n-1}(\mathbb{P}^n \times \mathbb{P}^{m-1})$. Therefore $Z$ is in fact a subvariety of $\mathbb{P}^n \times \mathbb{C}^m$. Since $Z$ is a projective variety, the map $Z \to \mathbb{C}^m$ induced by the second projection $\mathbb{P}^n \times \mathbb{C}^m \to \mathbb{C}^m$ is constant \cite[Corollary 5.2.2, Chapter I]{Shafarevich}. It follows that $Z$ is of the form $\mathbb{P}^n \times \mathbb{P}^0$.
\end{proof}

Lastly, we need the following mixed volume computation of Shephard \cite[pp. 134--136]{Shephard}.

\begin{lemma}[Shephard]\label{Shephard}
For positive numbers $\lambda_1 \ge \lambda_2 \ge \cdots \ge \lambda_n$, define the polytope
\[
\Delta_\mathbf{\lambda} = \text{conv}\big\{\mathbf{0},\lambda_1 \mathbf{e}_1,\lambda_2 \mathbf{e}_2,\ldots,\lambda_n \mathbf{e}_n\big\} \subset \mathbb{R}^n.
\]
If $\Delta \subset \mathbb{R}^n$ is the standard $n$-dimensional simplex, then 
\[
\text{MV}_n(\underbrace{\Delta,\ldots,\Delta}_{n-i},\underbrace{\Delta_\mathbf{\lambda},\ldots,\Delta_\mathbf{\lambda}}_i)
= \lambda_1 \lambda_2 \cdots \lambda_i \quad \text{for $0 \le i \le n$}.
\]
\end{lemma}

\begin{proof}[Proof of Theorem \ref{B}]

The case when $\xi\in A_*(\mathbb{P}^{n} \times \mathbb{P}^{m})$ is an integer multiple of either $[\mathbb{P}^{n} \times \mathbb{P}^{m}], [\mathbb{P}^{n} \times \mathbb{P}^{0}],[\mathbb{P}^{0} \times \mathbb{P}^{m}]$, or $[\mathbb{P}^{0} \times \mathbb{P}^{0}]$ is dealt with in Lemma \ref{exceptional}. Hereafter we assume that this is not the case. Let $\xi$ be the homology class
\[
\xi = \sum_i e_i \big[\mathbb{P}^{k-i} \times \mathbb{P}^{i}\big] \in A_k(\mathbb{P}^n \times \mathbb{P}^m),
\]
where the term containing $e_i$ is zero if $n<k-i$ or $m<i$. We write $H_1$ and $H_2$ for divisors on $\mathbb{P}^n \times \mathbb{P}^m$ obtained by pulling back a hyperplane from the first and the second factor respectively. 

\vspace{2mm}

\noindent 1. Suppose $\xi$ is represented by a subvariety $Z$. Then
\[
\int_{Z} \underbrace{H_1 \cdot \ldots \cdot H_1}_{k-i} \cdot \underbrace{H_2 \cdot \ldots \cdot H_2}_{i} = e_i.
\]
Since $H_1$ and $H_2$ are nef, Lemma \ref{TK} says that the $e_i$ form a log-concave sequence.

\vspace{2mm}

\noindent 2. We continue to assume that $\xi$ is represented by a subvariety $Z$. By Kleiman's theorem, $H_1|_Z$ and $H_2|_Z$ are limits of ample classes in the N\'eron-Severi space of $Z$ over $\mathbb{R}$ \cite[Theorem 1.4.23]{Lazarsfeld}. Therefore the sequence $\{e_i\}$ is a limit of log-concave sequences of positive real numbers. It follows from Lemma \ref{no internal zeros} that $\{e_i\}$ is a log-concave sequence of nonnegative numbers with no internal zeros. $\{e_i\}$ cannot be identically zero because, for example, $\sum_i {k \choose i} e_i$ is the degree of $Z$ inside the Segre embedding $\mathbb{P}^n \times \mathbb{P}^m \subset \mathbb{P}^{nm+n+m}$ \cite[Exercise 19.2]{Harris}.

\vspace{2mm}

\noindent 3. Now we show that the condition on the sequence is sufficient for the representability of a multiple of the corresponding homology class.
First we represent a multiple of $\xi \in A_n(\mathbb{P}^n \times \mathbb{P}^n)$, $n>0$, corresponding to a log-concave sequence of \emph{positive} integers $e_i$. Write
\[
\xi = \sum_{i=0}^n e_i \big[\mathbb{P}^{n-i} \times \mathbb{P}^{i} \big] \in A_n(\mathbb{P}^n \times \mathbb{P}^n).
\]
Let $e$ be a common multiple of $e_0,\ldots,e_{n-1}$ and define positive integers 
\[
\lambda_i = e (e_i/e_{i-1}) \quad \text{for $i=1,\ldots,n$}.
\]
The log-concavity of $e_i$ now reads $\lambda_1 \ge  \lambda_2 \ge \cdots \ge \lambda_{n}$. Let $\Gamma_\mathbf{\lambda}$ be the graph of the rational map 
\[
\varphi_\mathbf{\lambda} : \mathbb{P}^n \dashrightarrow \mathbb{P}^n, \quad 
 (z_0:\cdots:z_n) \longmapsto (z_0^{\lambda_1} : z_1^{\lambda_1} : z_0^{\lambda_1-\lambda_2} z_2^{\lambda_2}:\cdots : z_0^{\lambda_1-\lambda_n} z_n^{\lambda_n}).
\]
The projective degrees of $\varphi_\mathbf{\lambda}$ are the mixed multiplicities of the irrelevant ideal $\mathfrak{m}$ and the monomial ideal $J_\lambda$ generated by the components of $\varphi_\mathbf{\lambda}$. Combining Theorem \ref{mixed_volume} and Lemma \ref{Shephard}, we have
\[
e_i(\mathfrak{m}|J_\mathbf{\lambda})
= \text{MV}_n(\underbrace{\Delta,\ldots,\Delta}_{n-i},\underbrace{\Delta_\mathbf{\lambda},\ldots,\Delta_\mathbf{\lambda}}_i)
= \lambda_1 \lambda_2 \cdots \lambda_i = e^i(e_i/e_0),
\]
where $\Delta_\mathbf{\lambda}$ is the polytope of Lemma \ref{Shephard}. In other words, 
\[
\int_{\Gamma_\mathbf{\lambda} } \underbrace{H_1 \cdot \ldots \cdot H_1}_{n-i} \cdot \underbrace{H_2 \cdot \ldots \cdot H_2}_{i} = e^i(e_i/e_0).
\]
Now consider a regular map $\psi : \mathbb{P}^n \to \mathbb{P}^n$ defined by homogeneous polynomials $h_0,\ldots,h_n$ of degree $e$ with no common zeros. If the $h_i$ are chosen in a sufficiently general way, then the product $\psi \times \text{Id}_{\mathbb{P}^n} : \mathbb{P}^n \times \mathbb{P}^n \to \mathbb{P}^n \times \mathbb{P}^n$ restricts to a birational morphism $\Gamma_{\mathbf{\lambda}} \to \text{Im}(\Gamma_{\mathbf{\lambda}})$. From the projection formula we have
\[
\int_{\text{Im}(\Gamma_\mathbf{\lambda})} \underbrace{H_1 \cdot \ldots \cdot H_1}_{n-i} \cdot \underbrace{H_2 \cdot \ldots \cdot H_2}_{i} =\int_{\Gamma_\mathbf{\lambda}} \underbrace{e H_1 \cdot \ldots \cdot e H_1}_{n-i} \cdot \underbrace{H_2 \cdot \ldots \cdot H_2}_{i}  = e^n(e_i/e_0).
\]
In sum, $\text{Im}(\Gamma_{\mathbf{\lambda}}) \subset \mathbb{P}^n \times \mathbb{P}^n$ is irreducible and
\[
\big[\text{Im}(\Gamma_\mathbf{\lambda})\big] = (e^n/e_0) \sum_{i=0}^n e_i \big[\mathbb{P}^{n-i} \times \mathbb{P}^i \big] \in A_n(\mathbb{P}^n \times \mathbb{P}^n).
\]

\vspace{2mm}

\noindent 4. Finally, we represent a positive integer multiple of $\xi \in A_k(\mathbb{P}^n \times \mathbb{P}^m)$ corresponding to a nonzero log-concave sequence of nonnegative integers with no internal zeros. $\xi$ can be uniquely written as an integral linear combination
\[
\xi = \sum_{i=p}^q e_i \big[ \mathbb{P}^{k-i} \times \mathbb{P}^{i}\big] \in A_k(\mathbb{P}^n \times \mathbb{P}^m),
\]
where $0 \le p \le q \le k$, $k-p \le n$, $q \le m$, and $e_p, e_q$ are positive. 

If $p=q$, then either $0<p<m$ or $0<k-p<n$, since we are excluding the four exceptional cases. If $0<p<m$, take a hypersurface $Z$ in $\mathbb{P}^{k-p} \times \mathbb{P}^{p+1}$ defined by an irreducible bihomogeneous polynomial of degree $(0,e_p)$. Then the image of $Z$ under an embedding of $\mathbb{P}^{k-p} \times \mathbb{P}^{p+1}$ into  $\mathbb{P}^{n} \times \mathbb{P}^{m}$ represents $\xi$. Here $0<p$ guarantees the existence of the irreducible polynomial and $p<m$ guarantees the existence of the embedding. Similarly, if $0<k-p<n$, we take a hypersurface $Z$ in $\mathbb{P}^{k-p+1} \times \mathbb{P}^{p}$ defined by an irreducible bihomogeneous polynomial of degree $(e_p,0)$. The image of $Z$ under an embedding of $\mathbb{P}^{k-p+1} \times \mathbb{P}^{p}$ into $\mathbb{P}^{n} \times \mathbb{P}^{m}$ represents $\xi$. 

If $p<q$, then we can use the result of the previous step to choose a $(q-p)$-dimensional subvariety $Z \subset \mathbb{P}^{q-p} \times \mathbb{P}^{q-p}$ representing a multiple of 
\[
\sum_{i=p}^{q} e_{i} \big[\mathbb{P}^{q-i} \times \mathbb{P}^{i-p}\big] \in A_{q-p}(\mathbb{P}^{q-p} \times \mathbb{P}^{q-p}).
\]
Embed $\mathbb{P}^{q-p} \times \mathbb{P}^{q-p}$ into $\mathbb{P}^{k-p} \times \mathbb{P}^{q}$ and take the cone $\widetilde{Z}$ of $Z$ in $\mathbb{P}^{k-p} \times \mathbb{P}^{q}$. The cone $\widetilde{Z}$ is defined by the same bihomogeneous polynomials defining $Z$, hence irreducible, and represents a multiple of
\[
\sum_{i=p}^q e_i \big[ \mathbb{P}^{k-i} \times \mathbb{P}^{i}\big] \in A_k(\mathbb{P}^{k-p} \times \mathbb{P}^{q}).
\]
The image of $\widetilde Z$ under an embedding of  $\mathbb{P}^{k-p} \times \mathbb{P}^{q}$ into $\mathbb{P}^{n} \times \mathbb{P}^{m}$ represents a multiple of $\xi$.
\end{proof}

\section*{Acknowledgments}

The author would like to express deep gratitude to Professor Heisuke Hironaka for invaluable lessons on mathematics in general. He thanks Professor Hal Schenck for answering numerous questions on arrangements of hyperplanes and Professor Bernard Teissier for helpful comments on the general idealistic Bertini theorem. He also thanks the referees for useful suggestions.


\begin{thebibliography}{99}

\bibitem{Aigner} M. Aigner,
			\emph{Whitney numbers}, Combinatorial Geometries, 139--160, 
			Encyclopedia Math. Appl., 29, Cambridge Univ. Press, Cambridge, 1987. 
			MR0921072 

\bibitem{Aluffi} P. Aluffi,
			\emph{Computing characteristic classes of projective schemes},
			J. Symbolic Comput., {\bf 35} (2003), no. 1, 3--19. 
			MR1956868 (2004b:14007)

\bibitem{Brenti} F. Brenti, 
			\emph{Log-concave and unimodal sequences in algebra, combinatorics, and geometry: an update},
			Jerusalem Combinatorics '93,
			Contemp. Math. {\bf 178}, Amer. Math. Soc., Providence, RI, 1994, pp. 71--89.
			MR1310575 (95j:05026)

\bibitem{Bruns-Herzog} W. Bruns and J. Herzog, 
			\emph{Cohen-Macaulay Rings},
			Cambridge Studies in Advanced Mathematics, {\bf 39},
			Cambridge University Press, Cambridge, 1993.
			MR1251956 (95h:13020) 

\bibitem{Cox-Little-Oshea} D. Cox, J. Little, and D. O'Shea, 
			\emph{Using Algebraic Geometry},
			Graduate Texts in Mathematics, {\bf 185},
			Springer, New York, 1998.
			MR2122859 (2005i:13037)

\bibitem{Dimca} A. Dimca,
			\emph{Singularities and Topology of Hypersurfaces},
			Universitext, Springer-Verlag, New York, 1992.
			MR1194180 (94b:32058)
			
\bibitem{Dimca-Papadima} A. Dimca and S. Papadima, 
			\emph{Hypersurface complements, Milnor fibers and higher homotopy groups of arrangements}, 
			Ann. of Math. (2) {\bf 158} (2003), no. 2, 473--507.
			MR2018927 (2005a:32028) 


\bibitem{Fulton} W. Fulton,			
			\emph{Introduction to Toric Varieties},
			Annals of Mathematics  Studies, {\bf 131},
			Princeton University Press, Princeton, NJ, 1993.
			MR1234037 (94g:14028)

\bibitem{Gaffney1} T. Gaffney,
			\emph{Integral closure of modules and Whitney equisingularity},
			Invent. Math. {\bf 107} (1992), no. 2, 301--322. 
			MR1144426 (93d:32055)

\bibitem{Gaffney2} T. Gaffney,
			\emph{Multiplicities and equisingularity of ICIS germs},
			Invent. Math. {\bf123} (1996), no. 2, 209--220. 
			MR1374196 (97b:32051) 

\bibitem{Gromov} M. Gromov,
			\emph{Convex sets and K\"ahler manifolds},
			Advances in Differential Geometry and Topology,
			World Sci. Publ., Teaneck, NJ, 1990, pp. 1--38.
			MR1095529 (92d:52018) 
			
\bibitem{Harris} J. Harris,
			\emph{Algebraic Geometry. A First Course},
			 Corrected reprint of the 1992 original.,
			Graduate Texts in Mathematics, {\bf 133},
			Springer-Verlag, New York, 1995.
			MR1416564 (97e:14001) 

\bibitem{Hartshorne} R. Hartshorne,
			\emph{Varieties of small codimension in projective space}, 
			Bull. Amer. Math. Soc., {\bf 80} (1974), 1017--1032. 
			MR0384816 (52:5688)
						
\bibitem{Heron} A. P. Heron,
			\emph{Matroid polynomials},
			Combinatorics (Proc. Conf. Combinatorial Math., Math. Inst., Oxford, 1972),
			Inst. of Math. and its Appl., Southend-on-Sea, 1972, pp. 164--202.
			MR0340058 (49:4814) 
			
\bibitem{Swanson-Huneke} C. Huneke and I. Swanson, 
			\emph{Integral Closure of Ideals, Rings, and Modules}, 
			London Mathematical Society Lecture Note Series, {\bf 336},
			Cambridge University Press, Cambridge, 2006.
			MR2266432 (2008m:13013)
			
\bibitem{Jouanolou} J.-P. Jouanolou,
			\emph{Th\'eor\`emes de Bertini et Applications},
			Progress in Mathematics, {\bf 42}, Birkh\"auser Boston, Inc., Boston, MA, 1983.
			MR0725671 (86b:13007)

\bibitem{Kaveh-Khovanskii} K. Kaveh and A. G. Khovanskii,
			\emph{Newton-Okounkov bodies, semigroups of integral points, graded algebras and intersection theory},
			Preprint: {\bfseries\texttt arXiv:0904.3350v2}.
		
\bibitem{Khovanskii} A. G. Khovanskii, 
			\emph{Algebra and mixed volumes},
			Appendix 3 in: Yu. D. Burago and V. A. Zalgaller,
			Geometric Inequalities, Translated from the Russian by A. B. Sosinskii, 
			Grundlehren der Mathematischen Wissenschaften, {\bf 285}, 
			Springer Series in Soviet Mathematics, Springer-Verlag, Berlin, 1988.
			MR0936419 (89b:52020) 

\bibitem{Kouchnirenko}  A. G. Kouchnirenko, 
			\emph{Poly\`edres de Newton et nombres de Milnor},
			Invent. Math. {\bf 32} (1976), no. 1, 1--31.
			MR0419433 (54:7454) 

\bibitem{Kung} J. P. S. Kung,
			\emph{The geometric approach to matroid theory},
			Gian-Carlo Rota on Combinatorics, 604--622, 
			Contemp. Mathematicians, Birkh\"auser Boston, Boston, MA, 1995. MR1392975 

\bibitem{Lazarsfeld} R. Lazarsfeld,
			\emph{Positivity in Algebraic Geometry I},
			Ergebnisse der Mathematik und ihrer Grenzgebiete. 3. Folge. A Series of Modern Surveys in Mathematics, {\bf 48}, Springer-Verlag, Berlin, 2004.
			MR2095471 (2005k:14001a)
			
\bibitem{Lazarsfeld-Mustata} R. Lazarsfeld and M. Musta\c t\u a,
			\emph{Convex bodies associated to linear series},
			Ann. Sci. \'Ecole Norm. Sup\'er. (4) {\bf 42} (2009), no. 5, 783--835.
			MR2571958 (2011e:14012) 

\bibitem{MacPherson} R. D. MacPherson,			
			\emph{Chern classes for singular algebraic varieties},
			Ann. of Math. (2) {\bf 100} (1974), 423--432.
			MR0361141 (50:13587)

\bibitem{Marker} D. Marker,
			\emph{Model Theory. An Introduction}, 
			Graduate Texts in Mathematics, {\bf 217}. Springer-Verlag, New York, 2002.
			MR1924282 (2003e:03060) 
			
\bibitem{Northcott-Rees} D. G. Northcott and D. Rees, 
			\emph{Reduction of ideals in local rings},
			Proc. Cambridge Phil. Soc. {\bf 50} (1954), 145--158.
			MR0059889 (15,596a)

\bibitem{Okounkov} A. Okounkov, 
			\emph{A Brunn-Minkowski inequality for multiplicities}, 
			Invent. Math. {\bf 125} (1996), no. 3, 405--411. 
			MR1400312 (99a:58074) 

\bibitem{Oxley}	J. Oxley,
			\emph{Matroid Theory},
			Oxford Science Publications, Oxford University Press, New York, 2011.
			MR1207587 (94d:05033) 
			
\bibitem{Orlik-Solomon} P. Orlik and L. Solomon, 
			\emph{Combinatorics and topology of complements of hyperplanes}, 
			Invent. Math. {\bf 56} (1980), no. 2, 167--189.
			MR0558866 (81e:32015) 

\bibitem{Orlik-Terao} P. Orlik and H. Terao, 
			\emph{Arrangements of Hyperplanes}, 
			Grundlehren der Mathematischen Wissenschaften, {\bf 300}, Springer-Verlag, Berlin, 1992.
			MR1217488 (94e:52014)

\bibitem{Randell} R. Randell, 
			\emph{Morse theory, Milnor fibers and minimality of hyperplane arrangements}, 
			Proc. Amer. Math. Soc. {\bf 130} (2002), no. 9, 2737--2743.
			MR1900880 (2003e:32048) 

\bibitem{Read} R. C. Read, 
			\emph{An introduction to chromatic polynomials}, 
			J. Combinatorial Theory {\bf 4} (1968), 52--71.
			MR0224505 (37:104)

\bibitem{Rees-Sharp} D. Rees and R.Y. Sharp, 
			\emph{On a theorem of B. Teissier on multiplicities of ideals in local rings}, 
			J. London Math. Soc. (2) {\bf 18} (1978), no. 3, 449--463.
			MR0518229 (80e:13009) 

\bibitem{Rota} G. -C. Rota,			
			\emph{Combinatorial theory, old and new},
			Actes du Congr\`es International des Math\'ematiciens (Nice, 1970), Tome 3,
			Gauthier-Villars, Paris, 1971, pp. 229--233.
			MR0505646 (58:21703) 
			
\bibitem{Samuel} P. Samuel, 
			\emph{La notion de multiplicit\'e en alg\`ebre et g\'eom\'etrie alg\'ebrique}, 
			J. Math. Pures Appl. (9) {\bf 30} (1951), 159--274.
			MR0048103 (13,980c)

\bibitem{Schneider} R. Schneider,
			\emph{Convex Bodies: The Brunn-Minkowski Theory},
			Encyclopedia of Mathematics and its Applications, {\bf 44}, Cambridge University Press, Cambridge, 1993.
			MR1216521 (94d:52007) 

\bibitem{Shafarevich} I. R. Shafarevich,
			\emph{Basic Algebraic Geometry. 1. Varieties in Projective Space}, Second edition,
			Springer-Verlag, Berlin, 1994.
			MR1328833 (95m:14001)

\bibitem{Shephard} G. C. Shephard,
				\emph{Inequalities between mixed volumes of convex sets},
				Mathematika, {\bf 7} (1960), 125--138.	
				MR0146736 (26:4256) 		
			
\bibitem{Stanley} R. P. Stanley, 
			\emph{Log-concave and unimodal sequences in algebra, combinatorics, and geometry}, 
			Graph Theory and Its Applications: East and West (Jinan 1986),
			Ann. New York Acad. Sci., {\bf 576}, 1989, pp. 500--535.
			MR1110850 (92e:05124)

\bibitem{StanleyRota} R. P. Stanley,
			\emph{Foundations I and the development of algebraic combinatorics}, 
			Gian-Carlo Rota on Combinatorics, 105--107, 
			Contemp. Mathematicians, Birkh\"auser Boston, Boston, MA, 1995. 
			MR1392967 

\bibitem{Stanley2} R. P. Stanley, 
			\emph{Positivity problems and conjectures in algebraic combinatorics}, 
			Mathematics: Frontiers and Perspectives,
			Amer. Math. Soc., Providence, RI, 2000, pp. 295--319.
			MR1754784 (2001f:05001) 
		
\bibitem{Stanley3} R. P. Stanley,
			\emph{An introduction to hyperplane arrangements},
			Geometric Combinatorics, IAS/Park City Math. Ser., {\bf 13},
			Amer. Math. Soc., Providence, RI, 2007, pp. 389--496.
			MR2383131 	

\bibitem{Teissier} B. Teissier, 
			\emph{Cycles \'evanescents, sections planes et conditions de Whitney}, 
			Singularit\'es \`a Carg\`ese (Rencontre Singularit\'es G\'eom. Anal., Inst. E\'tudes Sci., Carg\`ese, 1972),
			Ast\'erisque, Nos. 7 et 8, Soc. Math. France, Paris, 1973, pp. 285--362.
			MR0374482 (51:10682)

			
\bibitem{Teissier2} B. Teissier, 
			\emph{Appendix: Sur une in\'egalit\'e \`a la Minkowski pour les multiplicit\'es}, in: 
			D. Eisenbud and H. Levine,
			\emph{An algebraic formula for the degree of a $C^\infty$ map germ},
			Ann. of Math. (2) {\bf 106} (1977), no. 1,  38--44.
			MR0467800 (57:7651)

\bibitem{Teissier3} B. Teissier, 
			\emph{Du th\'eor\`eme de l'index de Hodge aux in\'egalit\'es isop\'erim\'etriques},
			C. R. Acad. Sci. Paris Ser. A-B {\bf 288} (1979), no. 4, 287--289.
			MR0524795 (80k:14014)

\bibitem{Teissier4} B. Teissier, 
			\emph{Vari\'et\'es polaires. II. Multiplicit\'es polaires, sections planes, et conditions de Whitney}, 
			Algebraic Geometry (La R\'abida, 1981), 
			Lecture Notes in Math., {\bf 961}, Springer, Berlin, 1982, pp. 314--491. 
			MR 708342 (85i:32019)

\bibitem{Trung} N. V. Trung, 
			\emph{Positivity of mixed multiplicities},
			Math. Ann. {\bf 319} (2001), no. 1, 33--63.
			MR1812818 (2001m:13042) 

\bibitem{Trung-Verma} N. V. Trung and J. K. Verma, 
			\emph{Mixed multiplicities of ideals versus mixed volumes of polytopes}, 
			Trans. Amer. Math. Soc. {\bf 359} (2007), no. 10, 4711--4727.
			MR2320648 (2008e:13029) 

			
\bibitem{Welsh} D. Welsh,
		\emph{Matroid Theory},
		London Mathematical Society Monographs, {\bf 8}, Academic Press, London-New York, 1976.
		MR0427112 (55:148)
			
\end{thebibliography}
\end{document}